\documentclass[12pt, a4paper]{article}
\usepackage{amsfonts}
\usepackage{amsmath}
\usepackage[latin1]{inputenc}
\usepackage{amsthm}
\usepackage{amssymb}
\usepackage{pdfsync}
\usepackage[francais,english]{babel}
\usepackage{graphics,amssymb,epsfig,stmaryrd,color}
\numberwithin{equation}{section}
\theoremstyle{plain}
\newtheorem{thm}{Theorem}[section]
\newtheorem{coro}[thm]{Corollary}
\newtheorem{prop}[thm]{Proposition}
\newtheorem{lem}[thm]{Lemma}
\newtheorem{defi}[thm]{Definition}
\theoremstyle{definition}

\theoremstyle{remark}

\newcommand{\R}{\mathbb{R}}
\newcommand\ind{{\bf{1}}}
\newcommand{\PP}{\mathbb{P}}

\newcommand\N{{\mathbb N}}

\newcommand\pref[1]{(\ref{#1})}

\let \eps\varepsilon

\newcommand\CC{{\cal C}}

\newcommand\A{{\cal A}}

\newcommand\M{{\cal M}}

\newcommand\E{{\mathbb E}}

\def\<#1,#2>{\left<#1,#2\right>}

\newcommand\xx{\mathbf{X}}
\newcommand\yy{\mathbf{Y}}

\newcommand\LL{{\cal L}}
\newcommand\FF{{\cal F}}

\def\Prim{({\cal P})}
\def\Qrim{({\cal Q})}
\def\PPst{({\cal P}^*)}

\begin{document}

\title{Pareto efficiency for the concave order and multivariate
comonotonicity }
\author{G. Carlier \thanks{{\scriptsize CEREMADE, UMR CNRS 7534,
Universit\'e Paris IX Dauphine, Pl. de Lattre de Tassigny, 75775 Paris Cedex
16, FRANCE \texttt{carlier@ceremade.dauphine.fr}}}, R.-A. Dana \thanks{%
{\scriptsize CEREMADE, UMR CNRS 7534, Universit\'e Paris IX Dauphine, Pl. de
Lattre de Tassigny, 75775 Paris Cedex 16, FRANCE \texttt{%
dana@ceremade.dauphine.fr}}}, A. Galichon \thanks{
{\scriptsize D\'epartement d'Economie, Ecole Polytechnique, \texttt{\
alfred.galichon@polytechnique.edu}. Galichon gratefully acknowledges
support from Chaire EDF-Calyon ``Finance and D\'{e}veloppement Durable,'' Chaire  FDIR ``Socially Responsible Investments and Value Creation'' and Chaire Axa ``Assurance et Risques Majeurs'' and FiME,
Laboratoire de Finance des March\'{e}s de l'Energie (www.fime-lab.org).}}}
\maketitle

\begin{abstract}
  This paper studies efficient risk-sharing rules  for
the concave dominance order. For a univariate risk, it follows from a \emph{comonotone dominance principle}, due to Landsberger and Meilijson \cite{Landsberger}, that efficiency is
characterized by a comonotonicity condition. The goal of the paper is to generalize the comonotone dominance principle as well as the equivalence between efficiency and comonotonicity to the
multi-dimensional case. The multivariate case is more involved (in particular because there is no
immediate extension of the notion of comonotonicity), and it is addressed  by using
techniques from convex duality and optimal transportation.
\end{abstract}


\section{Introduction}

\textbf{Motivation. }The aim of this paper is to study Pareto efficient allocations of risky consumptions of
multiple goods in a contingent exchange economy. In this framework, consumption goods are imperfect substitutes, hence consumption is measured along several different units, instead of being denominated in one single monetary value. These units can be for instance
material consumption and labor, or future consumptions at various subsequent dates, or
currency units with limited exchangeability. In this setting, risky consumption can no longer be represented as a \emph{%
random variable}, but as a \emph{random vector}.

Agents are assumed to have incomplete preferences associated with the
 \emph{concave order}: a risk (random vector) $X$ is preferred to a risk $Y$ in the concave order whenever \emph{every} risk-averse expected utility decision-maker prefers $X$ to $Y$. Again, these preferences form an incomplete order, hence this assumption (and its empirical content) may appear as relatively weak. We shall see, however, that they in fact lead to strong predictions.

\smallskip

The motivation of the paper is to characterize efficient allocations for the concave order on observable data (for instance, insurance contracts).  In the case of  univariate risk, it is known that efficiency for the concave order is equivalent to efficiency for some strictly concave expected utility model, which in turn yields a tractable characterization of efficiency:  \emph{comonotonicity}  of the allocations. Allocations are comonotone whenever each agent's contingent consumption is a nondecreasing function of the aggregate consumption. Further, a \emph{comonotone dominance principle} can also be proven: if some initial allocation is not comonotone, there is a comonotone allocation such that every agent weakly prefers their contingent consumption in the new allocation (at least one preferring strictly). Comonotonicity fully characterizes efficiency and it is a testable and tractable property. Moreover, as a consequence of the  comonotone dominance principle, attention may be restricted to the set of    comonotone allocations, which is
convex and almost compact. Hence  existence results
may be  obtained for many risk-sharing problems  (see for instance \cite{jst} in the framework
of risk measures, or \cite{cd2}, \cite{cd3} for classes of law invariant and
concave utilities).

 \smallskip

\textbf{Main results. } This paper  is devoted to the extension of the  comonotone dominance result and its application to the characterization of efficient allocations  in the multivariate setting. To this end, a definition of comonotonicity  for the multivariate case is first needed.
Roughly speaking, according to the definition of multivariate comonotonicity we adopt, an allocation $(X_{1},...,X_{p})$ (with
each $X_{i}$ being random vectors) is comonotone if it is efficient for some strictly concave expected utility model.
By first order conditions, this implies that there is a random vector $Z$ and convex functions $\varphi _{i} $ such that
$X_{i}=\nabla \varphi _{i}(Z)$, where $\nabla \varphi _{i}$ is the gradient map of $ \varphi _{i}$. This is the definition of multivariate comonotonicity used by Ekeland, Galichon and Henry in \cite{EGH09}.   The comonotone dominance principle  is next extended by solving  a variational problem. More precisely, given an initial allocation and a collection of strictly concave utility functions, we   maximize the sum of these utilities among allocations that dominate the initial allocation. We   prove, and this is   the hard part of the proof, that the corresponding optimal allocation   is necessarily comonotone. The precise   statement  of the multivariate  comonotone dominance result is, however more
complicated than in the univariate case since  it  requires the use of   weak
closures   and   a  concept  slightly stronger than strict convexity. This follows from the fact that the set of multivariate comonotone
allocations is neither  convex nor compact (even up
to constants),   contrary to the univariate case (counterexamples are given). Finally one may wonder whether the equivalence between efficient and comonotone allocations is preserved in the multivariate case. The answer is yes, up to some (interesting) technicalities. Again the    precise   statement of the result  is more complicated than in the univariate case.
When applied to the  univariate case, our proof of the  comonotone dominance result  improves upon all existing proofs    (see \cite{Landsberger},   \cite {DanaMeilijson} and \cite{Rusch}). Indeed, it addresses directly the case of many agents, it uses neither the discrete case nor a limiting argument, and no hypotheses
need be made on the aggregate endowment.

\textbf{Literature overview. } There is a distinguished tradition in modeling  preferences by  concave dominance.   Introduced in economics  by
Rothschild and Stiglitz \cite{Rothschild}, the
 concave order has then been
used in a wide variety of economic contexts. To give a few references, let us mention efficiency
pricing (Peleg and Yaari \cite{Peleg}, Chew and  Zilcha \cite{Chew}), measurement of
inequality  (Atkinson \cite{Atkinson}), and finance (Dybvig \cite{Dybvig}, Jouini
and Kallal \cite{Jouini}).

\smallskip

In dimension one, the mutuality principle arose in the early work of  Borch \cite{Borch},  Arrow
\cite{Arrow}, \cite{Arrow2} and Wilson \cite{wilson}; see also LeRoy and Werner \cite{leroy}. Landsberger and Meilijson \cite{Landsberger} proved (for two agents and  a discrete setting) that any allocation of a given aggregate risk is dominated in the sense of
 concave dominance  by a comonotone allocation. This comonotone dominance principle has been extended   to the continuous case by limiting arguments (see \cite {DanaMeilijson} and \cite{Rusch}). It implies the comonotonicity of efficient allocations for the concave order. The equivalence between comonotonicity and efficiency was only proved recently by Dana \cite{dana} for the discrete case and by Dana and Meilijson  \cite{DanaMeilijson} for the continuous case. This equivalence stimulated a line of research on comonotonicity in the insurance and finance literature, see for instance Jouini and Napp \cite{jn1}, \cite{jn2}.  On the empirical side, Townsend \cite{townsend} proposed to test whether the mutuality principle holds in three poor villages in southern India while Attanasio and Davis (\cite{Attanasio}) worked with US labor data. The general findings of these empirical studies is that comonotonicity can be usually strongly rejected. A possible explanation of why efficiency is usually not observed in the data is that the aforementioned literature only considers risk-sharing in the case of one good (monetary consumption) and does not take into account the cross-subsidy effects between several risky goods which are only imperfect substitutes. Other papers, such as Brown and Matzkin (\cite{brown}) have tried to test whether observed market data on prices, aggregate endowments and individual incomes satisfy the restrictions that are imposed by Walrasian equilibrium. In contrast to this approach, we do not assume prices to be available to the researcher.

\smallskip

The notion of multivariate comonotonicity adopted in this paper coincides (up to some technical details) with the one originally introduced by
Ekeland, Galichon and Henry in \cite{EGH09} under the name $\mu $%
-comonotonicity, in the context of risk measures. Galichon and Henry use that concept to generalize rank-dependent expected utility in \cite{GH11}. Other proposals for multivariate comonotonicity exist and are reviewed e.g. in \cite{scarsini}; however they do not seem to be related to efficient risk-sharing.
While the results of \cite{EGH09} are strongly related to maximal
correlation functionals and to the quadratic optimal transportation problem
(and in particular Brenier's seminal paper \cite{brenier}), the present
approach will rely on a slightly different optimization problem that has
some familiarities with the multi-marginals optimal transport problem of
Gangbo and \'{S}wi\c{e}ch \cite{gs}.

\smallskip

\textbf{Organization of the paper. }The paper is organized as follows. Section \ref{prel} recalls some definitions
and various characterizations of comonotonicity in the univariate case. Section \ref{secuniv} revisits  the comonotone dominance principle of \cite{Landsberger} and characterizes efficient risk sharing in the univariate case.  A notion of
multivariate comonotonicity is introduced in Section \ref{secmult}, an
analogue of the comonotone dominance principle is stated, and efficient
sharing-rules are   characterized as the weak closure of comonotone allocations.  Section \ref{seccl} concludes the paper. Proofs are gathered in section \ref{proofs}.

\section{Preliminaries}

\label{prel}

Given as primitive is a probability space $(\Omega ,\mathcal{F},\mathbb{P})$. For every (univariate or multivariate) random vector $X$ on such space,
the law of $X$ is denoted $\mathcal{L}(X)$. Two random vectors $X$ and $Y$ are called equivalent in distribution (denoted $X\sim Y$), if $\mathcal{L}(X)=\mathcal{L}(Y)$.

\begin{defi}
Let  $X$ and $Y$ be bounded random vectors with values in $\mathbb{R}^{d}$,
then $X$ dominates $Y$  for the concave order,   denoted $X \succcurlyeq Y$, if and only if ${\mathbb{E}}(\varphi (X))\leq {%
\mathbb{E}}(\varphi (Y))$ for every convex function $\varphi \;:\;\mathbb{R}%
^{d}\rightarrow \mathbb{R}$. If, in addition, ${\mathbb{E}}(\varphi (X))<{%
\mathbb{E}}(\varphi (Y))$ for some  convex function $\varphi $,
then $X$ is said to dominate $Y$ strictly.
\end{defi}

As the paper    makes  extensive use of  \emph{convex} analysis  (Legendre transforms, infimal convolutions, convex duality), the  concave order is  defined here  in terms of convex
 loss functions while
usually defined with concave utilities. Clearly the definition above coincides with the
standard one. As
  $X \succcurlyeq Y$ implies that ${\mathbb{E}}(X)={\mathbb{E}}(Y)
$,   comparing risks for $ \succcurlyeq $ only makes sense
for random vectors with the same mean. We refer to   Rothschild and Stiglitz \cite%
{Rothschild}  and   F\"{o}llmer and  Schied \cite{FolSch} for various characterizations of concave
dominance in the univariate case and to   M\"{u}ller and  Stoyan \cite{Mul} for the multivariate case. Using a classical result of Cartier, Fell and Meyer (see \cite{cfm} or \cite{strassen}), one deduces a convenient characterization (see section \ref{proofs} for a proof) of strict dominance as follows:


\begin{lem}\label{caractstrict}
Let $X$ and $Y$ be bounded random vectors with values in $\mathbb{R}^{d}$, then the following statements are equivalent:

\begin{enumerate}
\item $X$ strictly dominates $Y$, \item $X \succcurlyeq Y$ and $\LL(X)\neq \LL(Y)$,
\item $X \succcurlyeq Y$ and for every strictly convex function $\varphi,\; $  ${\mathbb{E}}(\varphi (X))<{%
\mathbb{E}}(\varphi (Y))$.
\end{enumerate}

\end{lem}
Given $X\in L^{\infty }(\Omega ,\mathbb{R}^{d})$, a random vector of
aggregate risk of dimension $d\geq 1$, the
  set of admissible allocations or risk-sharing of $X$ among $p$
agents is denoted $\mathcal{A}(X)$:
\begin{equation*}
\mathcal{A}(X):=\{\mathbf{Y}=(Y_{1},...,Y_{p})\in L^{\infty }(\Omega ,%
\mathbb{R}^{d})\mbox{ : }\sum_{i=1}^{p}Y_{i}=X\}.
\end{equation*}%
For simplicity, the dependence of $\mathcal{A}%
(X)$ on the number $p$ of agents does not appear explicitly.   A concept of dominance for allocations of $X$ is defined next.
\begin{defi}For $d\geq 1$, let $\xx=(X_1,..., X_p)$ and $\mathbf{Y}:=(Y_1,...,Y_p)$ be in $\mathcal{A}%
(X)$. Then $\xx$ is said to dominate $\mathbf{Y}$ if $X_i \succcurlyeq Y_i$
for every $i\in \{1,...,p\}$. If, in addition, there is an $i\in \{1,...,p\}$
such that $X_i$ strictly dominates $Y_i$, then $\xx$ is said to strictly
dominate $\mathbf{Y}$. An allocation $\xx\in \mathcal{A}(X)$ is  Pareto-efficient (for the concave order) if there is no allocation in $%
\mathcal{A}(X)$ that strictly dominates $\xx$.

\end{defi}

It may easily be verified that dominance of allocations can also be defined as follows. Let  $\xx%
 $ and $\mathbf{Y} $ be in $\mathcal{A}%
(X) $, then $\xx$ dominates $\mathbf{Y}$ if and only if
\begin{equation}\label{sum}
{\mathbb{E}}(\sum_{i=1}^p\varphi _{i}(X_{i}))\leq {\mathbb{E}}(\sum_{i=1}^p\varphi
_{i}(Y_{i})),
\end{equation}%
for every collection of convex functions $\varphi _{i}$ : $\mathbb{R}%
^{d}\rightarrow \mathbb{R}$. Moreover, $\xx$ strictly dominates $\mathbf{Y}$
if and only if the previous inequality is strict for some collection of
  convex functions $\varphi _{i}$ : $\mathbb{R}^{d}\rightarrow
\mathbb{R}$ . Note that from lemma \ref{caractstrict}, it is equivalent to require that the inequality is strict for every collection of strictly convex functions.
 Therefore, if   $\xx$ is the solution of the problem
 \begin{equation}\label{sum1}
 \inf \Big\{ \sum_{i=1}^p \E(\varphi_i (Y_i)) \; : \;  (Y_1,...,Y_p)\in \A(X)\Big\}\end{equation}
  for some collection of
 strictly convex functions $\varphi _{i}$, then $\xx$ is efficient. Finally, recall that  in the univariate case, comonotonicity is defined by:
\begin{defi}\label{defcomon1d}
A collection  $(X_{1},...,X_{p})$  of $p$ real-valued random variables on $(\Omega ,%
\mathcal{F},\mathbb{P})$ is comonotone if for every $(i,j)\in\{1,...,p\}^2$,
\begin{equation*}
(X_{i}(\omega ^{\prime })-X_{i}(\omega ))(X_{j}(\omega ^{\prime
})-X_{j}(\omega ))\geq 0%
\mbox{ for $\PP\otimes \PP$-a.e. $(\omega,
\omega')\in \Omega^2$}.
\end{equation*}%
\end{defi}
It is well-known that comonotonicity of $(X_{1},...,X_{p})$ is equivalent to
the fact that each $X_i$ can be written as a nondecreasing function of the
sum $X=\sum_i X_i$ (see for instance Denneberg \cite{denneberg}). Therefore $(X_1,..., X_p)$ is comonotone if and only if there are nondecreasing functions $f_i$ summing to the
identity such that $X_i=f_i(X)$. Note  that the functions $f_i$ are
  all $1$-Lipschitz . The
extension of this notion to the multivariate case (i.e when each $X_i$ is $%
\mathbb{R}^d$-valued) is not immediately obvious and will be addressed in Section \ref%
{secmult}.

We now provide another characterization of comonotonicity based on the notion of maximal correlation.
From now on, assume that the underlying probability space $(\Omega ,\mathcal{F},\mathbb{P})$ is non-atomic which
means that there is no $A\in \mathcal{F}$ such that for every $B\in \mathcal{%
F}$ if $\mathbb{P}(B)<\mathbb{P}(A)$ then $\mathbb{P}(B)=0$. It is
well-known that $(\Omega ,\mathcal{F},\mathbb{P})$ is non-atomic if and only
if a random variable that is uniformly distributed on $\left[ 0,1\right]$, which is denoted
  $U\sim \mathcal{U}\left( \left[ 0,1\right] \right)$, can be constructed on $(\Omega ,\mathcal{F},\mathbb{P})$. Let $Z\in L^{1}(\Omega ,\mathcal{F},\mathbb{P})$,
and define for every $X\in L^{\infty }(\Omega ,\mathcal{F},\mathbb{P})$, both $Z$ and $X$ being univariate
here, the maximal correlation functional:
\begin{equation}\label{defcormax}
\varrho _{Z}(X):=\sup_{\tilde{X}\sim X}{\mathbb{E}}(Z\tilde{X})=\sup_{\tilde{%
Z}\sim Z}{\mathbb{E}}(\tilde{Z}X)=\sup_{\tilde{Z}\sim Z,\;\tilde{X}\sim X}{%
\mathbb{E}}(\tilde{Z}\tilde{X}).
\end{equation}
The functional $\varrho _{Z}$ has extensively been discussed in economics
and in finance, therefore only a few useful facts are recalled. Let $F_{X}^{-1}$
be the quantile function of $X$, that is the pseudo-inverse of distribution
function $F_{X}$. From Hardy-Littlewood's inequality, one has
\begin{equation*}
\varrho _{Z}(X)=\int_{0}^{1}F_{X}^{-1}(t)F_{Z}^{-1}(t)dt,
\end{equation*}%
and the supremum in \pref{defcormax} is achieved by any pair $(\tilde{Z},\tilde{X})$ of
comonotone random variables $(F_{Z}^{-1}(U),F_{X}^{-1}(U))$ for $U$
uniformly distributed. By symmetry, one can either fix $Z$ or fix $X$.
Fixing for instance $Z$, the supremum is achieved by $F_{X}^{-1}(U)$ where $%
U\sim \mathcal{U}\left( \left[ 0,1\right] \right) $ and $%
Z=F_{Z}^{-1}(U)$. When $Z$ is non-atomic, there is a unique $U=F_{Z}(Z)$
such that $Z=F_{Z}^{-1}(U)$, and the supremum is uniquely attained by the
non-decreasing function of $Z,\;F_{X}^{-1}\circ F_{Z}(Z)$:
\begin{equation}
\varrho _{Z}(X)={\mathbb{E}}(ZF_{X}^{-1}\circ F_{Z}(Z)).  \label{HL}
\end{equation}%
Also note that  $\varrho _{Z}$ is subadditive: $\varrho _{Z}(\sum_i  X_i) \leq \sum_i \varrho _{Z }( X_i)$.

\begin{prop}
\label{maxcor} Let $(X_1,...,X_p)$ be in $L^{\infty}(\Omega ,\mathcal{F},\mathbb{P})$. The following assertions are equivalent:

\begin{enumerate}
\item $(X_1,..., X_p)$ are comonotone,

\item for any non-atomic $Z\in L^{1}(\Omega ,\mathcal{F},\mathbb{P})$,
\begin{equation}  \label{sadd}
\varrho _{Z}\left(\sum_i  X_i\right) = \sum_i \varrho _{Z }\left( X_i\right),
\end{equation}

\item for some non-atomic $Z\in L^{1}(\Omega ,\mathcal{F},\mathbb{P})$, (\ref{sadd}) holds true.
\end{enumerate}
\end{prop}

\begin{proof} For the sake of simplicity, we restrict ourselves to $p=2$ and set $(X_1,X_2)=(X,Y)$. Point 1 implies point 2 since $F^{-1}_{X+ Y}=F^{-1}_X+F^{-1}_Y$  for  comonotone $X$ and $Y$. To show that point 3 implies point 1, assume that for some non-atomic $Z$, one has (\ref{sadd}), which by sublinearity is equivalent to   $\varrho _{Z } ( X+Y)\geq \varrho _{Z } ( X )+\varrho _{Z } ( Y)$.  Let $Z_{X+Y}$ (resp. $ Z_{X }$ and $ Z_{Y}$) be distributed as $Z$ and solve     $\sup_{\tilde{Z}\sim Z }E (\tilde Z  X )$ (resp. $\varrho _{Z } ( X )$ and $\varrho _{Z } ( Y)$). One then has:
$$\E(Z_{X+Y}(X+Y))\geq \E(Z_{X }X)+\E(Z_{Y }Y).$$
As $\E(Z_{X+Y}X)\leq \E(Z_{X }X)$ and $\E(Z_{X+Y}Y)\leq \E(Z_{Y }Y)$, it follows $\E(Z_{X+Y}X)= \E(Z_{X }X)=\varrho _{Z } ( X)$ and $\E(Z_{X+Y}Y)= \E(Z_{Y }Y)=\varrho _{Z } ( Y)$,  hence  from (\ref{HL}),  $X=F^{-1}_X\circ F _{Z_{X+Y}}({Z_{X+Y}}))$ and $Y=F^{-1}_Y\circ F _{Z_{X+Y}}({Z_{X+Y}}))$, proving comonotonicity.\end{proof}

Proposition \ref{maxcor} was the starting point of Ekeland, Galichon and
Henry \cite{EGH09} for providing a multivariate generalization of the
concept of comonotonicity. In the sequel we shall further discuss this
multivariate extension and compare it with the one proposed in the present
paper.

\section{The univariate case}

\label{secuniv}

A landmark result, due to Landsberger and Meilijson \cite%
{Landsberger} states that any allocation is dominated by a comonotone one.
The original proof was given in the discrete case for two agents, and the
results were extended to the general case by approximation. We give an alternative proof  in the Appendix  based on
the same approach we shall use in the multidimensional case. This proof is based
on a certain optimization problem; we believe that, even in the unidimensional case, it is of interest per se since
it does not require approximation arguments and slightly improves on the
original statement by proving strict dominance of non-comonotone
allocations. Contrary to Landsberger and Meilijson, one needs however to assume, as before, that the
probability space $(\Omega ,\mathcal{F},\mathbb{P})$ is non-atomic.

\begin{thm}
\label{dominance1d} Let $X$ be a bounded real-valued random variable on the
non-atomic probability space $(\Omega ,\mathcal{F},\mathbb{P})$, and let $\xx%
=(X_{1},...,X_{p})\in \mathcal{A}(X)$ be an allocation. There exists a
comonotone allocation in $\mathcal{A}(X)$ that dominates $\xx$. Moreover, if
$\xx$ is not comonotone, then there exists   an allocation that strictly
dominates $\xx$.
\end{thm}

As an application, we have:

\begin{thm}
\label{equivce} Let $X$ be a bounded real-valued random
variable on the non-atomic probability space $(\Omega ,\mathcal{F},\mathbb{P}%
)$ and let $\xx=(X_1,...,X_p) \in \A(X)$. Then the following statements are equivalent:

\begin{enumerate}

\item $\xx$ is efficient,

\item $\xx$ is comonotone,

\item there exist continuous and strictly convex functions $(\psi_1,...,\psi_p)$ such that $\xx$ solves
\[\inf \{ \sum_{i=1}^p \E(\psi_i(Y_i)) \; : \; \sum_{i=1}^p Y_i=X\}.\]


\end{enumerate}

\end{thm}\begin{proof}
Point $1$ implies point $2$: the comonotonicity of efficient allocations of $X$  follows directly  from Theorem \ref{dominance1d}. Point 2 implies point 3: if $\xx=(X_1,..., X_p)$ is comonotone, let us write $X_i=f_i(X)$ for some nondecreasing and $1$-Lipschitz functions $f_i$: $[m,M]\to \R$ (with  $M:={\rm{Esssup}} X$, $m:={\rm{Essinf}} X$) summing up to the identity map. Extending the $f_i$ functions by $f_i(x)=f_i(M)+(x-M)/p$ for $x\geq M$ and $f_i(x)=f_i(m)+(x-m)/p$ for $x\leq m$, one gets $1$-Lipschitz nondecreasing functions summing up to the identity everywhere. Let $\varphi(x):=\int_0^x f_i(s)ds$ for every $x$. The functions $\varphi_i$ are  convex and $C^{1,1}$ (i.e. $C^1$ with a Lipschitz continuous derivative)  and have quadratic growth at $\infty$. The convex conjugates\footnote{Let us recall that the Legendre transform or convex conjugate of $\varphi_i$ is by definition given by $\varphi_i^*(x):=\sup_{y} \{ x \cdot y -\varphi_i(y)\}$.} $\psi_i:=\varphi_i^*$ are strictly convex and continuous functions, and by construction, one has for every $i$, $X\in \partial \psi_i(X_i)$ a.s., which implies that $(X_1,..., X_p)$ minimizes $\E(\sum_i \psi_i (Y_i))$ subject to $\sum_ i Y_i=X$, which proves point $3$. Point $3$ implies point $1$ since the $\psi_i$ functions are strictly convex; if $(X_1,..., X_p)$ satisfies point $3$ then it is an efficient allocation of $X$. 

\end{proof}

\begin{coro}
Let $(\Omega ,\mathcal{F},\mathbb{P})$ be non-atomic, then  the set of efficient allocations of $X$
is convex and compact in $L^{\infty}$ up to zero-sum translations (which
means that it can be written as $\{(\lambda_1,..., \lambda_p) \; : \;
\sum_{i=1}^p \lambda_i=0 \}+A_0$ with $A_0$ compact in $L^{\infty}$). In
particular, the set of efficient allocations of $X$ is closed in $L^{\infty}$.
\end{coro}

\begin{proof}
Let $M:={\rm{Esssup}} X$, $m:={\rm{Essinf}} X$ and define $K_0$ as the set of functions $(f_1,..., f_p)\in C([m,M], \R^p)$ such that for each nondecreasing $f_i$, $f_i(m)=m/p$ and $\sum_{i=1}^p f_i(x)=x$ for every $x\in [m, M]$, and let
\[K:=K_0+\{(\lambda_1,..., \lambda_p) \; : \; \sum_{i=1}^p \lambda_i=0 \}.\]
The convexity claim thus  follows from theorem \ref{equivce}  and the convexity of $K$. Let us remark that elements of $K_0$ have $1$-Lipschitz components and are bounded.  The compactness of  $K$ in $C([m,M], \R^p)$ then follows from Ascoli's theorem. The compactness and closedness claims  directly follow.
\end{proof}

 Convexity and compactness of efficient allocations are   quite remarkable
features and as will be shown later,  they are no longer true in the multivariate
case. Note also that efficient allocations are regular: they are $1$%
-Lipschitz functions of aggregate risk.

\section{The multivariate case}

\label{secmult}

The aim of this section is to generalize to the multivariate case the results obtained in the
univariate case. More particularly,  Landsberger and Meilijson's comonotone dominance principle are extended: 1) any allocation is dominated by a comonotone allocation; 2) any non comonotone allocation is strictly dominated by a comonotone one.

When addressing these generalizations it is not immediately clear what is the appropriate notion of
comonotonicity  in the multivariate framework. Let us informally give an intuitive presentation of the approach developed in the following paragraphs. A natural generalization of monotone maps in several dimensions is given by subgradients of convex functions. It is therefore tempting  to say that an allocation $(X_1,..., X_p)\in \A(X)$ is comonotone whenever there is a common random vector $Z$ (interpreted as a price) and convex functions $\varphi_i$ (interpreted as individual costs) such that $X_i \in \partial \varphi_i (Z)$ a.s. for every $i$. Formally, this is nothing but the optimality condition for the risk-sharing or infimal convolution problem
\begin{equation}\label{rsr}
\inf_{\xx\in \A(X)} \sum_{i=1}^p \E(\psi_i (X_i)),
\end{equation}
where $\psi_i =\varphi_i^*$ (the Legendre Transform of $\varphi_i$). This suggests a definition of comonotone allocations as the allocations that solve a risk-sharing problem of the type above. This has a natural interpretation in terms of risk-sharing, but  one has to be cautious about such a definition whenever the $\psi_i$ functions are degenerate\footnote{In the univariate case, the situation is much simpler since one can take $Z=X$, and since the $X_i$ variables sum up to $X$, each convex function $\varphi_i$ has to be differentiable i.e. all the $\psi_i$  necessarily are strictly convex. In other words, degeneracies can be ruled out easily in the univariate case.}. Indeed, if all the $\psi_i$ functions are constant, then any allocation is comonotone in that sense! This means that  one   has to impose strict convexity in the definition.  We shall actually go one step further in \emph{quantifying} strict convexity as follows. Given an arbitrary collection $w=(w_1,..., w_p)$ of strictly convex  functions, we will say that an allocation is $w$-strictly comonotone whenever it solves a risk-sharing problem of the form \pref{rsr} for some  $\psi_i$ functions which are \emph{more convex} than the $w_i$ (i.e. $\psi_i -w_i$ is convex for every $i$). Allocations which can be approached (in law) by strictly $w$-comonotone will   be called comonotone. Since they solve a strictly convex risk-sharing problem, $w$-strictly comonotone allocations are efficient and the main goal of this section will be to generalize the univariate comonotone dominance result. We shall indeed prove that for any allocation $\xx\in \A(X)$  and any choice of $w$, there is a $w$-comonotone allocation $\yy\in \A(X)$  that dominates $\xx$ (strictly whenever $\xx$ is not itself $w$-comonotone). The full proof is detailed in Section \ref{proofs}, but its starting point is quite intuitive and consists of studying the optimization problem:
\begin{equation}\label{pbmeprimalyy}
\inf \Big\{ \sum_{i=1}^p \E(w_i (Y_i)) \; : \;  (Y_1,...,Y_p)\in \A(X), \; Y_i \succcurlyeq X_i, \; i=1,...,p\Big\}.
\end{equation}
Clearly, the solution $\yy$ of \pref{pbmeprimalyy}  dominates $\xx$. A careful study of the dual of \pref{pbmeprimalyy} will enable us to prove that $\yy$ is necessarily  $w$-comonotone, thus  giving the desired multivariate extension of Landsberger and Meilijson's comonotone dominance principle. Note also, that our proof is constructive since it relies on an explicit (although difficult to solve in practice)  convex minimization problem.

\smallskip

This section is organized as follows. In paragraph \ref{reform}, we shall reformulate the problem in terms of joint laws rather than random allocations. This is purely technical but will enable us to gain some linearity and some compactness in \pref{pbmeprimalyy}. We then define precisely our concepts of multivariate comonotonicity in paragraph \ref{defeffcom}. Paragraph \ref{dommult}  states the multivariate comonotone dominance result, i.e. the multivariate generalization of Landsberger and Meilijson's results. Finally, in paragraph \ref{disccom}, we gather several remarks on multivariate comonotonicity and emphasize some important qualitative differences between the univariate and multivariate cases.

\subsection{From random vectors to joint laws}\label{reform}

From now on,  it is assumed that the underlying
probability space $(\Omega ,\mathcal{F},\mathbb{P})$ is non-atomic, that   there are $p$ agents and that risk
is $d$-dimensional.  $X$ is a
given $\mathbb{R}^{d}$-valued $L^{\infty }$ random vector modeling an
aggregate random multivariate risk, while $\xx=(X_{1},....,X_{p})$ is a
given $L^{\infty }$ sharing of $X$ among the $p$ agents, that is
\begin{equation*}
X=\sum_{i=1}^{p}X_{i}.
\end{equation*}%
Let $\gamma _{0}:=\mathcal{L}(\xx)$ be the joint law of $\xx$ and $m_{0}:=\mathcal{L}(X)$.
Let $\gamma $ be  a probability measure on $(\mathbb{R}^{d})^{p}$ and   $\gamma ^{i}$ denote its $i$-th marginal.
 Note that, $ \mathcal{L}(Y_{i})$ is the $i$-th marginal of $
\mathcal{L}(\mathbf{Y})$.   Let $\Pi _{\Sigma }\gamma $ be the probability measure on $\mathbb{R}^{d}$
defined by
\begin{equation}
\int_{\mathbb{R}^{d}}\varphi (z)d\Pi _{\Sigma }\gamma (z)=\int_{\mathbb{R}%
^{d\times p}}\varphi (\sum_{i=1}^{p}x_{i})d\gamma
(x_{1},...,x_{p}),\;\forall \varphi \in C_{0}(\mathbb{R}^{d},\mathbb{R}),
\label{defdepisigma}
\end{equation}%
(where $C_{0}$ denotes the space of continuous functions that tend to $0$ at $%
\infty $). It follows from this definition that if $\gamma =\mathcal{L}(%
\mathbf{Y})$, then $\Pi _{\Sigma }\gamma =\mathcal{L}(\sum Y_{i})$. Hence, if
$\mathbf{Y}\in \mathcal{A}(X)$ and $\gamma =\mathcal{L}(\mathbf{Y})$, then  $\Pi _{\Sigma }\gamma =m_{0}=\mathcal{L}(X)$. In other words, if $%
\gamma =\mathcal{L}(\mathbf{Y})$ with $\mathbf{Y}\in \mathcal{A}(X)$, then
\begin{equation}
\int \varphi (x_{1}+...+x_{d})d\gamma (x_{1},...,x_{d})=\int \varphi
(z)dm_{0}(z),\;\forall \varphi \in C_{0}(\mathbb{R}^{d},\mathbb{R}).
\label{loisum}
\end{equation}%
Since $\mathbf{Y}$
is bounded, $\gamma $ is compactly supported. It follows from the next lemma that $\{\mathcal{L}(\mathbf{Y}),\;\mathbf{Y}\in
\mathcal{A}(X)\}$ coincides with the set of compactly supported probability
measures $\gamma $ on $(\mathbb{R}^{d})^{p}$  that satisfy (\ref{loisum}):

\begin{lem}
\label{caractloi} Assume $(\Omega, \mathcal{F}, \mathbb{P})$ is non-atomic. If $%
\gamma$ is a compactly supported probability measure on $(\mathbb{R}^d)^p$
and satisfies (\ref{loisum}), then there exists a random vector $\mathbf{Y}%
=(Y_1,...,Y_p)\in \mathcal{A}(X)$ such that $\mathcal{L}(\mathbf{Y})=\gamma$. Hence $\{ \mathcal{L}(\mathbf{Y}), \; \mathbf{Y}\in \mathcal{A}(X)\}=\mathcal{M}%
(m_0)$,  where $\mathcal{M}(m_0)$ is the set of compactly supported probability
measures on $(\mathbb{R}^d)^p$ such that $\Pi_{\Sigma}
\gamma=m_0=\Pi_{\Sigma} \gamma_0$.
\end{lem}

In the sequel,   joint laws $\mathcal{M}(m_{0})$ will be used
instead of   admissible allocations $\mathcal{A}(X)$. For compactness issues, a closed ball $B\in \mathbb{R}^{d}$  centered at $0$ such that $m_{0}$ is supported by $B^{p}$  is chosen,
 and attention is restricted   to
the set of elements of $\mathcal{M}(m_{0})$ supported by $pB$ (meaning that only risk-sharings of $X$ whose components take value in $B$ will be considered). We thus define
\begin{equation*}
\mathcal{M}_{B}(m_{0}):=\{\gamma \in \mathcal{M}(m_{0})\;:\;\gamma
(B^{p})=1\}.
\end{equation*}

\subsection{Efficiency and comonotonicity in the multivariate case}

\label{defeffcom}

Let  $\mathcal{C}$ be the cone of convex and continuous functions on $B$, dominance and efficiency in terms of joint laws are defined as follows:

\begin{defi}
Let $\gamma$ and $\pi$ be in $\mathcal{M}_B(m_0)$, then $\gamma$ dominates $%
\pi$ whenever
\begin{equation}  \label{domingam}
\int_{B^p} \sum_{i=1}^p \varphi_i(x_i) d \gamma (x_1,..., x_p) \leq \int_{B^p}
\sum_{i=1}^p \varphi_i(x_i) d \pi (x_1,..., x_p)
\end{equation}
for all functions $(\varphi_1, ..., \varphi_p) \in \mathcal{C}^p$. If, in
addition, inequality (\ref{domingam}) is strict whenever the $%
\varphi_i$ functions are further assumed to be strictly convex, then $\gamma$ is said
to dominate strictly $\pi$. The allocation $\gamma\in \mathcal{M}_B(m_0)$ is
  efficient if there is no other allocation in $\mathcal{M}_B(m_0)$
that strictly dominates it.
\end{defi}

Given $\gamma_0\in \mathcal{M}_B(m_0)$, it is easy to check (taking
functions $\varphi_i(x) =\vert x_i \vert^n$ in \pref{domingam} and letting $n\to \infty$) that any
$\gamma\in \mathcal{M}(m_0)$ dominating $\gamma_0$ (without the restriction that it is supported on $B^p$) actually belongs to $\mathcal{M}%
_B(m_0)$. Hence the choice to only consider allocations supported by $B^p$
is in fact not restrictive. Indeed, if $\gamma$ is supported by $B^p$, then efficiency of $\gamma$ in the usual sense, i.e. without restricting to competitors supported by $B^p$, is \emph{equivalent} to efficiency among competitors supported by $B^p$.

To define comonotonicity, let $\psi
:=(\psi _{1},...,\psi _{p})$ be a family of strictly convex continuous
functions (defined on $B$).  For any $x\in pB$, let us consider the  risk sharing (or infimal convolution)
 problem:
\begin{equation*}
\Box _{i}\psi _{i}(x):=\inf \left\{ \sum_{i=1}^p\psi _{i}(y_{i})\;:\;y_{i}\in
B,\;\sum_{i=1}^py_{i}=x\right\} .
\end{equation*}%
This problem admits a unique solution which will be denoted
\begin{equation*}
T_{\psi }(x):=(T_{\psi }^{1}(x),...,T_{\psi }^{p}(x)).
\end{equation*}%
Note that, by definition
\begin{equation}
\sum_{i=1}^pT_{\psi }^{i}(x)=x, \;\forall
x\in pB.  \label{samesum}
\end{equation}%

The   map $x\mapsto T_{\psi }(x)$ gives the optimal way to share $%
x$ so as to minimize the total cost when each individual cost is $\psi _{i}$. It defines the efficient allocation $T_{\psi }(X):=(T_{\psi }^{1}(X),...,T_{\psi }^{p}(X))$  with joint law $\gamma _{\psi }$  defined by:
\begin{equation*}
\int_{B^{p}}f(y_{1},...,y_{p})d\gamma _{\psi }(y):=\int_{pB}f(T_{\psi
}(x))dm_{0}(x)
\end{equation*}%
for any $f\in C(B^{p})$. One then defines comonotonicity as follows:
\begin{defi}\label{defcom11}
An allocation $\gamma \in \mathcal{M}_{B}(m_{0})$ is strictly comonotone if
there exists a family $\psi :=(\psi
_{1},...,\psi _{p})$  of strictly convex continuous  functions such that $\gamma =\gamma _{\psi }$. Given a family $w :=(w _{1},...,w
_{p})$ of
strictly convex functions in $C^{1}(B)$, an allocation $\gamma \in \mathcal{M}_{B}(m_{0})$ is $w $%
-strictly comonotone if there exists a family $\psi :=(\psi _{1},...,\psi _{p})$ of convex continuous functions
such that $\psi _{i}-w _{i}\in
\mathcal{C}$ for every $i$ and $\gamma =\gamma _{\psi }$.
\end{defi}
We shall soon show that strictly comonotone random vectors are   in the image of monotone operators (subgradients of convex functions), evaluated at the same random vector,  $p(X)$, which justifies the terminology ``comonotonicity'' in the multivariate setting.
 By definition, any strictly comonotone allocation is efficient.  As the set of   strictly comonotone
 allocations is not  closed, we are led to introduce another definition.

\begin{defi}\label{defcom22}
An allocation $\gamma \in \mathcal{M}_B(m_0)$ is comonotone if there exists
a sequence of strictly comonotone allocations that weakly star converges to $%
\gamma$. Given a family $w:=(w_1,..., w_p)$ of strictly convex functions in $C^1(B)$, an allocation $\gamma \in \mathcal{M}%
_B(m_0)$ is $w$-comonotone, if there exists a sequence of $w$%
-strictly comonotone allocations that weakly star converges to $\gamma$.
\end{defi}

Definitions \ref{defcom11} and \ref{defcom22} will be discussed in more detail in paragraph \ref{disccom}.  To understand the previous notions of comonotonicity and in particular  why these allocations are called comonotone,  it is  important to
 understand   the structure of the $T_{\psi }$ maps.

Let us first
ignore regularity issues and further assume that the  $\psi _{i}$ functions
are smooth as well as $\psi _{i}^{\ast }$ their Legendre transforms. Without the constraints $x_{i}\in B$, then the optimality conditions   imply
that there is some multiplier $p=p(x)$ such that
\begin{equation*}
\nabla \psi _{i}(T_{\psi }^{i}(x))=p,\mbox{ hence, } \; T_{\psi }^{i}(x)=\nabla
\psi _{i}^{\ast }(p).
\end{equation*}%
Using (\ref{samesum}), one gets
\begin{equation*}
x=\sum_{j=1}^p\nabla \psi _{j}^{\ast }(p),\mbox{ hence, } \;p=\nabla (\sum_{j=1}^p\psi
_{j}^{\ast })^{\ast }(x),
\end{equation*}%
thus,
\begin{equation*}
T_{\psi }^{i}(x)=\nabla \psi _{i}^{\ast }\left( \nabla (\sum_{j=1}^p\psi
_{j}^{\ast })^{\ast }(x)\right) .
\end{equation*}%
The maps $T_{\psi }^{i}$ are therefore composed of gradients of convex functions
that sum up to the identity. In dimension $1$, gradients of convex functions
  are  simply monotone maps (and   so are composed of such maps). In higher
dimensions, a richer  and more complicated structure
emerges that will be discussed later. Let us now consider  the full problem with the constraints that $x_{i}\in B$ and
still assume that the $\psi _{i}$ functions are smooth, then the optimality
conditions read as the existence of a $p$ and a $\lambda _{i}\geq 0$ such
that $\nabla \psi _{i}(T_{\psi }^{i}(x))=p-\lambda _{i}T_{\psi }^{i}(x)$ holds together with the complementary slackness conditions: $\lambda _{i}=0$ whenever  $T_{\psi}^i(x)$ lies in the interior of $B$.

\subsection{A multivariate dominance result and equivalence between
efficiency and comonotonicity}

\label{dommult}

Let us fix an allocation $\xx=(X_1,...,X_p)\in \A(X)$ such that $\xx \in B^p$ a.s., and set $%
\gamma_0=\mathcal{L}(\xx)$ so that $\gamma_0\in \mathcal{M}_B(m_0)$.
A family $w:=(w_1,..., w_p)$ of $C^1$ functions is also given, each of them being strictly
convex on $B$ as in section \ref{defeffcom}. The first main result
in the multivariate case is a dominance result, it states that every
allocation is  dominated by a $w$-comonotone one and that the dominance is  strict if
the initial allocation is not itself $w$-comonotone.
\begin{thm}
\label{dominancedimd} Let $\gamma _{0}=\mathcal{L}(\xx)$ and $w $ be as
above. Then there exists some $\gamma \in \mathcal{M}(m_{0})$ that is $w $%
-comonotone and dominates $\gamma _{0}$. Moreover if $\gamma _{0}$ is not
itself $w $-comonotone, then $\gamma $ strictly dominates $\gamma _{0}$.
\end{thm}

The   proof of this result will be given in section \ref{proofs}. Without giving details at this point, let us explain the main arguments of the proof:

\begin{itemize}

\item The optimization problem \pref{pbmeprimalyy} admits a unique solution $\yy$ with law $\gamma=\LL(\yy)$, which is efficient  and dominates $\gamma_0=\LL(\xx)$.

\item One then proves that $\gamma$  is  necessarily $w$-comonotone, by showing that that $w$-comonotonicity is an  optimality condition for \pref{pbmeprimalyy}. As usual in convex programming, optimality conditions can be obtained by duality.  This  leads to consider the problem
\begin{equation}\label{pbmedualpsi}
\inf \Big\{  \E\Big(\sum_{i=1}^p \psi_i(X_i)-\Box_i \psi_i(\sum_{i=1}^p X_i) \Big) \; : \; \psi_i-w_i \mbox{ convex}, \; \forall i \Big\}.
\end{equation}
By a careful study of \pref{pbmedualpsi}, one can prove (but this is rather technical) that $\gamma$ is $w$-comonotone.

\item It remains to show that $\gamma$ strictly dominates $\gamma_0$ unless $\gamma_0$ is itself $w$-comonotone.   From lemma \ref{caractstrict}, it suffices to show that $\yy\neq \xx$. But if $\gamma_0$ is not $w$-comonotone, then $\xx$ cannot be optimal for \pref{pbmeprimalyy} and thus $\yy\neq \xx$.

\end{itemize}

In terms of
efficiency, the following thus holds:

\begin{thm}
\label{efficiency} Let $\gamma \in \mathcal{M}_{B}(m_{0})$ and $w $ be
as before. Then

\begin{enumerate}
\item if $\gamma$ is strictly $w$-comonotone, then it is  efficient,

\item if $\gamma$ is efficient, then it is $w$-comonotone for any $w$,

\item the closure for the weak-star topology of efficient
allocations coincides with the set of $w$-comonotone allocations (which is therefore independent of $w$).
\end{enumerate}
\end{thm}

\begin{proof}
Point $1$ is a property already mentioned several times. Point $2$ follows from Theorem \ref{dominancedimd} and point $3$ follows from points $1$ and $2$.

\end{proof}

 Note that by definition, if $\gamma_0$ is strictly $w$-comonotone then the value of problem \pref{pbmedualpsi} is zero.  We shall also prove (see section \ref{proofs}) weak form of the converse, namely that if  the value of problem \pref{pbmedualpsi} is zero then $\gamma_0$ is $w$-comonotone. Therefore, the value of \pref{pbmedualpsi}  as a function of the joint law $\gamma_0$ can be viewed as a numerical criterion for comonotonicity and thus for efficiency. One can therefore, in principle, use on data this value as a test statistic for efficiency.


\subsection{Remarks on multivariate comonotonicity}\label{disccom}

{\textbf{Comparison with the notion of $\mu$-comonotonicity of \cite{EGH09}}.}
The notion of multivariate comonotonicity considered in this paper is to be related
to the notion of $\mu $-comonotonicity proposed by Ekeland, Galichon and
Henry in \cite{EGH09}. Recall the alternative characterization of
comonotonicity given in the univariate case in Proposition \ref{maxcor}: $%
X_{1}$ and $X_{2}$ are comonotone if and only if $\varrho _{\mu }\left(
X_{1}+X_{2}\right) =\varrho _{\mu }\left( X_{1}\right) +\varrho _{\mu
}\left( X_{2}\right) $ for a measure $\mu $ that  is sufficiently regular. In dimension $d$, \cite{EGH09} have introduced the concept of $\mu $%
-comonotonicity, based on this idea: if $\mu $ is a probability measure on $%
\mathbb{R}^{d}$ which does not give positive mass to small sets, two random
vectors $X_{1}$ and $X_{2}$ on $\mathbb{R}^{d}$ are called $\mu $\emph{%
-comonotone} if and only if
\begin{equation*}
\varrho _{\mu }\left( X_{1}+X_{2}\right) =\varrho _{\mu }\left( X_{1}\right)
+\varrho _{\mu }\left( X_{2}\right),
\end{equation*}%
where the (multivariate) \emph{maximum correlation functional }(see e.g.
\cite{Ru06}\ or \cite{EGH09}) is defined by%
\begin{equation*}
\varrho _{\mu }\left( X\right) =\sup_{\tilde{Y}\sim \mu }{\mathbb{E}}\left(
X\cdot \tilde{Y}\right) .
\end{equation*}%
The authors of \cite{EGH09} show that $X_{1}$ and $X_{2}$ are $\mu $-comonotone if and only
if there are two convex functions $\psi _{1}$ and $\psi _{2}$, and a random
vector $U\sim \mu $ such that
\begin{equation*}
X_{1}=\nabla \psi _{1}\left( U\right) \text{ and }X_{2}=\nabla \psi
_{2}\left( U\right)
\end{equation*}%
holds almost surely. Therefore, the present notion of multivariate
comonotonicity approximately consists of calling $X_{1}$ and $X_{2}$
comonotone if and only if there is some measure $\mu $ such that $X_{1}$ and
$X_{2}$ are $\mu $-comonotone. There are, however, qualifications to be
added. Indeed, \cite{EGH09} require some regularity on the measure $\mu$. In the current setting, no regularity restrictions are imposed on $\mu $; but instead restrictions on the
convexity of $\psi _{1}$ and $\psi _{2}$ have to be imposed to define the notion of $w $%
-comonotonicity before passing to the limit. Although not equivalent, these
two sets of restrictions originate from the same concern: two random vectors
are always optimally coupled with very degenerate distributions, such as the
distribution of constant vectors. Therefore one needs to exclude these
degenerate cases in order to avoid a definition which would be void of
substance. This is the very reason why the strictly convex $w_i$ functions had to be introduced.

\smallskip

{\textbf{Comonotone allocations do not form a bounded set}.}
In the scalar case, comonotone allocations are
parameterized by the set of nondecreasing functions summing to the identity
map. This set of functions is convex and equilipschitz hence compact (up to
adding constants summing up to zero). This compactness is no longer true in
higher dimensions (at least when $w=0$ and we work on the whole space instead of $B$), and we believe that this is a major structural difference
with respect to the univariate case. For simplicity assume that $p=2$. As
outlined in paragraph \ref{defeffcom}, a comonotone allocation $(X_{1},X_{2})$ of $%
X$ is given by a pair of functions that are composed of gradients of convex
functions and sum up to the identity map. It is no longer true, in dimension $2$
that this set of maps is compact (up to constants). Indeed, let us take $n\in
{\mathbb{N}}^{\ast }$, and quadratic $\psi _{1}$ and $\psi _{2}$ of the form
\begin{equation*}
\psi _{i}(x)=\frac{1}{2}\<S_{i}^{-1}x,x>,\;i=1,2,\;x\in \mathbb{R}^{2}
\end{equation*}%
with
\begin{equation*}
S_{1}=\left(
\begin{array}{cc}
\frac{1}{2} & \frac{1}{8\sqrt{n}} \\
\frac{1}{8\sqrt{n}} & \frac{1}{2n}%
\end{array}%
\right) ,\hspace{1cm}S_{2}=\left(
\begin{array}{cc}
\frac{1}{2} & \frac{-1}{8\sqrt{n}} \\
\frac{-1}{8\sqrt{n}} & \frac{1}{2n}%
\end{array}%
\right).
\end{equation*}%
Then the corresponding map $T_{\psi }$ is linear, and $T_{\psi }^{1}$ is given by
the matrix
\begin{equation*}
S_{1}(S_{1}+S_{2})^{-1}=\left(
\begin{array}{cc}
\frac{1}{2} & \frac{\sqrt{n}}{8} \\
\frac{1}{8\sqrt{n}} & \frac{1}{2}%
\end{array}%
\right)
\end{equation*}%
which is unbounded.

\smallskip

{\textbf{Comonotone allocations do not form a convex set}.}
Another difference with the univariate case is that the set of maps of the
form $T_{\psi }$ used to define comonotonicity is not convex. To see this
(again in the case $p=d=2$), it is enough to show that the set of pairs of $%
2\times 2$ matrices
\begin{equation*}
K:=(S_{1}(S_{1}+S_{2})^{-1},S_{2}(S_{1}+S_{2})^{-1}),\;S_{i}%
\mbox{ symmetric, positive
definite},\;i=1,2\}
\end{equation*}%
is not convex. First let us remark that if $(M_{1},M_{2})\in K$ then $M_{1}$
and $M_{2}$ have a positive determinant. Now for $n\in {\mathbb{N}}^{\ast }$%
, and $\eps\in (0,1)$ consider
\begin{equation*}
S_{1}=\left(
\begin{array}{cc}
1 & \sqrt{1-\eps} \\
\sqrt{1-\eps} & 1%
\end{array}%
\right) ,\hspace{1cm}S_{2}=\left(
\begin{array}{cc}
1 & -\sqrt{1-\eps} \\
-\sqrt{1-\eps} & 1%
\end{array}%
\right) ,
\end{equation*}%
\begin{equation*}
S_{1}^{\prime }=\left(
\begin{array}{cc}
1 & \sqrt{n-\eps} \\
\sqrt{n-\eps} & n%
\end{array}%
\right) ,\hspace{1cm}S_{2}^{\prime }=\left(
\begin{array}{cc}
1 & -\sqrt{n-\eps} \\
-\sqrt{n-\eps} & n%
\end{array}%
\right) ,
\end{equation*}%
and define:
\[M_i=S_i (S_{1}+S_{2})^{-1}, \; M_{i}^{\prime }=S_{i}^{\prime }(S_{1}^{\prime
}+S_{2}^{\prime })^{-1}, \; i=1,2.\]
If $K$ was convex then the following matrix  would have a positive determinant
\begin{equation*}
M_{1}+M_{1}^{\prime }=\left(
\begin{array}{cc}
1 & \frac{\sqrt{1-\eps}}{2}+\frac{\sqrt{n-\eps}}{2n} \\
\frac{\sqrt{1-\eps}}{2}+\frac{\sqrt{n-\eps}}{2} & 1%
\end{array}%
\right) ,
\end{equation*}%
which is obviously false for $n$ large
enough and $\eps$ small enough.

\section{Concluding remarks}\label{seccl}

In this paper, we have extended   Landsberger and Meilijson's comonotone dominance principle to the multivariate case by introducing the  variational    problem \pref{pbmeprimalyy}. We have then extended the univariate theory of efficient
risk-sharing to the case of several goods without perfect substitutability,
and we derived tractable implications. Two observations can be made at this point. In the first place, this paper demonstrates \textit{the intrinsic difficulty of the multivariate case,} as many features of the univariate case do not extend to higher dimensions: computational ease, the
compactness and convexity of efficient risk-sharing allocations. Second, it illustrates \textit{the need for qualification} inherent to the multivariate case. Contrary to the univariate case, the need to quantify strict convexity as in this paper comes by no
coincidence. In fact, just as \cite{EGH09} impose regularity conditions
on their \textquotedblleft baseline measure\textquotedblright\ to avoid
degeneracy, we work with cones which are strictly included in the cone of
convex functions by quantifying the strict convexity of the functions used.

Getting back to our initial motivation, namely, finding testable implications  of efficiency for the concave order, we already emphasized in paragraph \ref{dommult} that one obtains as a byproduct of our variational approach a numerical criterion that could in principle be used as a test statistic for comonotonicity and thus for efficiency. We thus believe that the present work paves the way for an interesting research agenda. First of all, an efficient algorithm to decide whether a given allocation in the multivariate case is comonotone or not remains to be discovered -- we are currently investigating this point. The convex nature of the underlying optimization problem helps, but the constraints of problem $(\mathcal{P}^*)$ are delicate to handle numerically. Finally, this work opens a research agenda on the empirical relevance of the multivariate theory confronted to the data: do observations of realized allocations of risk satisfy restrictions imposed by multivariate comonotonicity? As mentioned above, tests in the univariate case have been performed by \cite{Attanasio} and \cite{townsend} and suggest rejection. But there is hope that in the more flexible setting of multivariate risks, efficiency would be less strongly rejected. 


\section{Proofs}

\label{proofs}

\subsection{Proof of Lemma \ref{caractstrict}}

Clearly $1\Rightarrow  2$ and $3\Rightarrow  1$ are obvious. To prove that $2\Rightarrow 3$,  assume that 2 holds true. Let  $\mu:=\LL(X)$ and  $\nu:=\LL(Y)$.  These   probability measures are supported by some closed ball $B$, and the Cartier-Fell-Meyer theorem states that there is a measurable family of conditional probability measures $(T_x)_{x\in B}$ such that $T_x$ has mean $x$ and for every $f$ continuous function, one has
\[\E(f(Y))=\int_B f(y) d\nu(y)=\int_B \int_B f(y) dT_x(y) d \mu(x)\]
Since $\mu\neq \nu$, $\mu (\{ x\in B \; : \;  T_x\neq \delta_x\})>0$, one deduces  from Jensen's inequality that  for every strictly convex function $\varphi$,  $\E(\varphi(Y))>\E(\varphi(X))$.

\subsection{Proof of Lemma \ref{caractloi}}

For notational simplicity, assume that $d=1$, $p=2$, $X$ takes values in $[0,2]$ a.s. (so that $m_0$ has support in $[0,2]$) and $\gamma$ is supported by $[0,1]^2$. For every $n\in \N^*$ and $k\in \{0,..., 2^{n+1}\}$, set
\[ X^n:=\sum_{k=0}^{2^{n+1}} \frac{k}{2^n} \ind_{A_k,n}, \; \mbox{ where } A_{k,n}:=\left\{\omega\in \Omega \; : \; X(\omega)\in \Big[\frac{k}{2^n}, \frac{k+1}{2^n}\Big[\right\},\]
and
\[C_{k,n}:=\left\{(y_1,y_2)\in[0,1]^2 \mbox{ : } y_1+y_2 \in \Big[\frac{k}{2^n}, \frac{k+1}{2^n}\Big[\right\}.\]
Decompose the strip $C_{k,n}$  into a partition by triangles
\[C_{k,n}=\bigcup_{k\leq i+j \leq k+1} T_{k,n}^{i,j}, \;  T_{k,n}^{i,j}:=C_{k,n} \cap \Big [\frac{i}{2^n}, \frac{i+1}{2^n}\Big[\times \Big [\frac{j}{2^n}, \frac{j+1}{2^n}\Big[.\]
Since $\Pi_{\Sigma}(\gamma)=m_0$ one has:
\[\PP(A_{k,n})=\gamma(C_{k,n}) =\sum_{k\leq i+j \leq k+1} \gamma(T_{k,n}^{i,j}),\]
and since $(\Omega, \FF, \PP)$ is non-atomic, it follows from Lyapunov's convexity theorem (see \cite{lyapunov}) that there exists a partition of $A_{k,n}$ into measurable subsets $A_{k,n}^{i,j}$ such that
\begin{equation}\label{lyap}
\gamma(T_{k,n}^{i,j})=\PP(A_{k,n}^{i,j}),  \forall (i,j)\in \{0,..., 2^n\} \; : \; k\leq i+j \leq k+1.
\end{equation}
Choose $(y_1,y_2)_{k,n}^{i,j} \in T_{k,n}^{i,j}$ and define
\[\yy^n=(Y_1^n, Y_2^n):=\sum_{k=0}^{2^{n+1}} \sum_{k\leq i+j \leq k+1} (y_1,y_2)_{k,n}^{i,j} \ind_{A_{k,n}^{i,j}}.\]
We may also choose inductively the partition of $A_{k,n}$ by the $A_{k, n}^{i,j}$ to be finer and finer with respect to $n$. By construction, one obtains
\[ \max\Big(\Vert X^n-X\Vert_{L^{\infty}}, \Vert X^n-Y_1^n-Y_2^n\Vert_{L^{\infty}}, \Vert \yy^{n+1}-\yy^n\Vert_{L^{\infty}}\Big) \leq \frac{1}{2^n},\]
so that $\yy^n$ is a Cauchy sequence in $L^{\infty}$, and thus converges to some $\yy=(Y_1,Y_2)$. One then sees that $Y_1+Y_2=X$,  and passing to the limit in \pref{lyap}, it follows that $\LL(\yy)=\gamma$.

\subsection{Proofs and variational characterization for the multivariate dominance result}\label{proof3}

The proofs will very much rely on the linear programming problem:
\begin{equation*}
(\mathcal{P}^*) \; \sup_{\gamma\in K(\gamma_0)} -\int_{B^p} \sum_{i=1}^p
w_i(x_i) d\gamma (x)
\end{equation*}
where $K(\gamma_0)$ consists of all $\gamma\in\mathcal{M}_B(m_0)$ such that
for each $i$ the marginal $\gamma^i$ dominates the corresponding marginal of
$\gamma_0$ i.e.:
\begin{equation*}
\int_{B^p} \varphi(x_i) d \gamma(x)\leq \int_{B^p} \varphi(x_i)
d\gamma_0(x), \forall \varphi \mbox{ convex on $B$}.
\end{equation*}
Problem $(\mathcal{P}^*)$ presents  similarities with the problem solved  in \cite{gs}. In
the optimal transport problem considered in \cite{gs}, one minimizes the
average of some quadratic function over joint measures having prescribed
marginals whereas $(\mathcal{P}^*)$ includes dominance constraints on the
marginals. To shorten notations, define
\begin{equation*}
\eta(x):=-\sum_{i=1}^p w_i(x_i), \forall x=(x_1,..., x_p)\in B^p
\end{equation*}

$(\mathcal{P}^*)$ is the dual problem (see the next lemma for details) of
\begin{equation*}
(\mathcal{P}) \inf \Big \{ \int_{B^p} \Big(\sum_{i=1}^p \varphi_i(x_i)-\varphi_0%
\Big(\sum_{i=1}^p x_i\Big) \Big)d \gamma_0(x), \; (\varphi_0,..., \varphi_p) \in E %
\Big\},
\end{equation*}
where $E$ consists of all families $\varphi:=(\varphi_1,...., \varphi_p,
\varphi_0)\in C(B)^p \times C(pB)$ such that $\varphi_i\in \mathcal{C}$ and
\begin{equation*}
\sum_{i=1}^p \varphi_i(x_i)- \varphi_0\Big(\sum_{i=1}^p x_i\Big) \geq -\sum_{i=1}^p
w_i(x_i).
\end{equation*}
It will also be convenient to consider
\begin{equation*}
(\mathcal{Q}) \inf \Big \{J(\psi) \;, \; \psi=(\psi_1,..., \psi_p)
\mbox{ :
each $\psi_i$  is such that $\psi_i-w_i$ is convex} \Big\}
\end{equation*}
with
\begin{equation*}
J(\psi):=\int_{B^p} \Big( \sum_{i=1}^p \psi_i(x_i)-\Box_i \psi_i \Big(\sum_{i=1}^p x_i%
\Big) \Big) d\gamma_0(x).
\end{equation*}
Note that by construction $J(\psi)\geq 0$ for every admissible $\psi$ and $%
J(\psi)=0$ if and only if $\gamma_0=\gamma_{\psi}$.

\begin{lem}
\label{dualite11} The following holds
\begin{equation*}
\max (\mathcal{P}^*)= \inf (\mathcal{P})=\inf (\mathcal{Q}) - \int_{B^p}
\sum_{i=1}^p w_i(x_i) d\gamma_0 (x).
\end{equation*}
\end{lem}

\begin{proof}
Let us write $\Prim$ in the form
\[\inf_{\varphi=(\varphi_1,..., \varphi_p, \varphi_0)\in C(B)^p \times C(pB)}  F(\Lambda \varphi)+G(\varphi)\]
where $\Lambda$  :  $C(B)^{p}\times C(pB) \to C(B^p)$ is the linear continuous map defined by
\[\Lambda \varphi(x):=\sum_{i=1}^p \varphi_i(x_i)-\varphi_0\Big(\sum_{i=1}^p x_i\Big), \; \forall x=(x_1,..., x_p)\in B^p,\]
and $F$, $G$ are the convex lower semicontinuous (for the uniform norm) functionals defined  by
\[F(\theta)=\left\{\begin{array}{lll}
\int_{B^p} \theta d \gamma_0  \mbox{ if } &  \theta \geq \eta\\
+\infty & \mbox{ otherwise }
\end{array}
\right., \; \forall \theta\in C(B^p)\]
\[G(\varphi)=\left\{\begin{array}{lll}
0 \mbox{ if } &  (\varphi_1,..., \varphi_p)\in \CC^p\\
+\infty & \mbox{ otherwise}
\end{array}
\right., \forall \varphi =(\varphi_1,..., \varphi_p, \varphi_0)\in C(B)^p \times C(pB).\]
It is easy to see that the assumptions of  Fenchel-Rockafellar's duality theorem (see for instance \cite{Ektem})  are satisfied and thus
\[\inf \Prim= \max_{\gamma \in \M(B^p) } -F^*(\gamma_0-\gamma)-G^*(\Lambda^*(\gamma-\gamma_0)).\]
The adjoint of $\Lambda$,  $\Lambda^*$ is easily computed as : $\M(B^p) \to \M(B)^p \times \M(pB)$ (where $\M$ denotes the space of Radon measures):
\[\Lambda^* \gamma=(\gamma^1, ..., \gamma^p, -\Pi_{\Sigma} \gamma), \; \forall \gamma \in \M(B^p).\]
Direct computations give
\[F^*(\gamma-\gamma_0)= \left\{\begin{array}{lll}
-\int_{B^p} \eta d \gamma \mbox{ if } &  \gamma \geq 0\\
+\infty & \mbox{ otherwise }
\end{array}
\right.\]
and
\[G^*(\Lambda^*(\gamma-\gamma_0))=\left\{\begin{array}{lll}
0 \mbox{ if } &  \gamma\in K(\gamma_0)\\
+\infty & \mbox{ otherwise.}
\end{array}
\right.\]
Therefore $\PPst$ is the dual of $\Prim$ in the usual sense of convex programming and $\max \PPst= \inf \Prim$. To prove that
\[\inf \Prim=\inf \Qrim - \int_{B^p} \sum_{i=1}^p w_i(x_i) d\gamma_0 (x),\]
take $\varphi\in E$ and $\psi_i :=w_i+\varphi_i$ for $i=1,.., p$, the constraint then reads as
\[\sum_{i=1}^p \psi_i(x_i)\geq \varphi_0\Big(\sum_{i=1}^p x_i\Big), \; \forall x\in B^p.\]
Now in $\Prim$, one needs to choose $\varphi_0$ as large as possible without violating this constraint. Thus the best $\varphi_0$ given $(\varphi_1,..., \varphi_p)$ is
\[\varphi_0=\Box_i \psi_i,\]
which proves the desired identity.

\end{proof}

\begin{lem}
\label{lemderiv} Let $\psi_i$ be such that $\psi_i-w_i\in \mathcal{C}$
for every $i$ and $g=(g_1,..., g_p)\in \mathcal{C}^p$. Then
\begin{equation*}
\begin{split}
\lim_{\delta\to 0^+} \frac{1}{\delta} [J(\psi+\delta g)-J(\psi)] = \sum_{i=1}^p
\int_{B} g_i(x_i) d (\gamma_0^i -\gamma_{\psi}^i) \\
=\int_{B^p} \sum_{i=1}^p g_i(x_i) d (\gamma_0-\gamma_{\psi})(x).
\end{split}%
\end{equation*}
\end{lem}

\begin{proof}
For $\delta>0$, one first gets that
\[\begin{split}
\frac{1}{\delta} [J(\psi+\delta g)-J(\psi)] = \sum_{i=1}^p \int_{B}  g_i(x_i) d (\gamma_0^i)-
\\\int_{pB} \frac{1}{\delta}   \Big( \Box_i (\psi_i+\delta g_i)(x)-\Box_i \psi_i (x) \Big)d m_0(x).
\end{split}\]
Note that the integrand in the second term is bounded since $g$ is. Now fix some $(x_1,..., x_p)\in B^p$, and set $x=\sum_{i=1}^p x_i$, $y_i:=T_{\psi}^i (x)$ and $y^{\delta}_i:=T_{\psi+\delta g}(x)$. Since $\sum_{i=1}^p y_i = \sum_{i=1}^p y_i^{\delta}=x$, it comes as a direct consequence of the definition of infimal convolutions that:
\begin{equation}\label{ineg41}
\frac{1}{\delta}   \Big( \Box_i (\psi_i+\delta g_i)(x)-\Box_i \psi_i (x) \Big) \leq \sum_{i=1}^p g_i (y_i)
\end{equation}
and
\begin{equation}\label{ineg42}
\frac{1}{\delta}   \Big( \Box_i (\psi_i+\delta g_i)(x)-\Box_i \psi_i (x) \Big) \geq \sum_{i=1}^p g_i (y_i^{\delta}).
\end{equation}
Using the compactness of $B$ and the strict convexity of $\psi_i$, it is easy to check that $y_i^{\delta} \to y_i$ as $\delta\to 0^+$. Therefore, from \pref{ineg41} and \pref{ineg42} one has
\[\lim_{\delta \to 0^+} \frac{1}{\delta}   \Big( \Box_i (\psi_i+\delta g_i)(x)-\Box_i \psi_i (x) \Big) = \sum_{i=1}^p g_i (T_{\psi}^i(x))\]
and this holds for every $x\in pB$. It then follows from Lebesgue's dominated convergence theorem that
\[\begin{split}
\lim_{\delta\to 0^+} \frac{1}{\delta} [J(\psi+\delta g)-J(\psi)] =  \sum_{i=1}^p \int_{B}  g_i(x_i) d (\gamma_0^i)-  \sum_{i=1}^p \int_{pB}  g_i(T_{\psi}^i(x)) dm_0(x) \\
=  \sum_{i=1}^p  \int_{B} g_i(x_i) d (\gamma_0^i -\gamma_{\psi}^i)= \int_{B^p} \sum_{i=1}^p g_i(x_i) d (\gamma_0-\gamma_{\psi})(x).
\end{split}\]

\end{proof}

It follows from Lemma \ref{lemderiv} that, if $\psi$ solves $(\mathcal{Q})$,
then $\gamma_{\psi}$ dominates $\gamma_0$. Hence, if one knew that $(\mathcal{Q%
})$ possesses solutions, the existence of an $\omega$-strictly comonotone
allocation dominating $\gamma_0$ would directly follow. Unfortunately, it
is not necessarily the case that the infimum in $(\mathcal{Q})$ is attained -- or
at least we haven't been able to prove without additional conditions. The
difficulty here comes from the fact that minimizing sequences need not be
bounded (see paragraph \ref{disccom}). It may be the case that additional regularity assumptions on $%
\gamma_0$ would guarantee existence. In the present paper no such assumption is made, and a different path is chosen to overcome the difficulty by an appeal to Ekeland's variational principle.

\begin{lem}
\label{ekelandvp} Letting $\eps>0$, there exists $\psi_\eps$ admissible for $(%
\mathcal{Q})$ such that

\begin{enumerate}
\item $J(\psi_\eps) \leq \inf (\mathcal{Q}) +\eps$

\item
\begin{equation*}
\limsup_{\eps \to 0^+} \int_{B^p} \sum_{i=1}^p \varphi_i(x_i) d
(\gamma_{\psi_\eps} -\gamma_0) \leq 0
\end{equation*}
for every $(\varphi_1,..., \varphi_p)\in \mathcal{C}^p$

\item
\begin{equation*}
\liminf_{\eps \to 0^+} \int_{B^p} \sum_{i=1}^p \varphi_i^\eps (x_i) d
(\gamma_{\psi_\eps} -\gamma_0) \geq 0
\end{equation*}
for $\varphi_i^{\eps}=\psi_{i, \eps}-w_i$ (these are convex functions
by definition).
\end{enumerate}
\end{lem}

\begin{proof}
For $\eps>0$, let  $f_\eps$ be admissible for $\Qrim$ and such that
\[J(f_\eps)\leq \inf \Qrim +\eps.\]
Let then $k_\eps>0$ be such that
\begin{equation}\label{condsurk}
\lim_{\eps \to 0^+}  \eps k_\eps[1 + \Vert f_\eps \Vert] =0 \; \mbox{ (for instance } k_\eps=\frac{1}{\eps^{1/2}(1+ \Vert f_\eps \Vert)}).
\end{equation}
It follows from Ekeland's variational principle (see \cite{Ekeland} and \cite{aubinek}) that  for every $\eps>0$, there is some $\psi_\eps$  admissible for $\Qrim$ such that
\begin{equation}\label{pve11}
\Vert \psi_\eps-f_\eps\Vert \leq \frac{1}{k_\eps}, \; J(\psi_\eps) \leq J(f_\eps) \leq \inf \Qrim +\eps,
\end{equation}
where $\Vert  h \Vert$  stands for the sum of the  uniform norms of the $h_i$ functions, and
\begin{equation}\label{pve12}
J(\psi) \geq J(\psi_\eps)-k_\eps \eps \Vert \psi-\psi_\eps\Vert, \; \forall \psi =(\psi_1,..., \psi_p) \; : \; \psi_i-w_i \in \CC, \; \forall i.
\end{equation}
Taking $\psi=\psi_\eps+\delta \varphi$ with $\delta>0$ and $\varphi\in \CC^p$ in \pref{pve12}, dividing by $\delta$ and letting $\delta\to 0^+$,  one thus gets by the virtues of Lemma \ref{lemderiv}
\begin{equation}\label{ineg321}
\int_{B^p} \sum_{i=1}^p \varphi_i(x_i) d (\gamma_0-\gamma_{\psi_\eps})\geq -k_\eps \eps \Vert \varphi \Vert.
\end{equation}
Using \pref{condsurk} and letting $\eps\to 0^+$ yields
\begin{equation}\label{almostdominance}
\limsup_{\eps \to 0^+} \int_{B^p} \sum_{i=1}^p \varphi_i(x_i) d (\gamma_{\psi_\eps} -\gamma_0) \leq 0
\end{equation}
for every $(\varphi_1,..., \varphi_p)\in \CC^p$.
To prove the last assertion of the lemma, write $\psi_\eps= \varphi^\eps+w$ with $\varphi^\eps\in \CC^p$. Then for $\delta \in (0,1)$ one has $\psi_\eps-\delta \varphi^\eps=(1-\delta) \varphi^\eps+w$, and then \pref{pve12} can be applied to $\psi_\eps-\delta \varphi^\eps$, leading to
\[ \frac{1}{\delta} [J (\psi_\eps-\delta \varphi^\eps)-J(\psi_\eps)] \geq -k_\eps \eps \Vert \varphi^\eps \Vert,\]
which, letting $\delta\to 0^+$ and using the same argument as in lemma \ref{lemderiv}, leads in turn to
\[ \int_{B^p} \sum_{i=1}^p \varphi^{\eps}_i(x_i) d (\gamma_{\psi_\eps}-\gamma_0)  \geq -k_\eps \eps \Vert \varphi^\eps\Vert.\]
By \pref{condsurk} and \pref{pve11}, it follows that
\[ k_\eps \eps \Vert \varphi^\eps\Vert \leq  k_\eps \eps (\Vert w\Vert +\Vert \psi_\eps-f_\eps\Vert + \Vert f_\eps\Vert)  \leq   k_\eps \eps  \Vert w\Vert + \eps + k_\eps \eps \Vert f_\eps\Vert \to 0 \mbox{ as } \eps\to 0^+,\]
which enables us to conclude that
\begin{equation}\label{slackapprox}
\liminf_{\eps \to 0^+} \int_{B^p} \sum_{i=1}^p \varphi_i^\eps (x_i) d (\gamma_{\psi_\eps} -\gamma_0) \geq  0.
\end{equation}
\end{proof}

\begin{lem}
\label{lemcvgg} Let $\psi_\eps$ be as in lemma \ref{ekelandvp} and set $%
\gamma_\eps:=\gamma_{\psi_\eps}$ then up to some subsequence, $\gamma_\eps$
weakly star converges to some $\gamma$ ($w$-comonotone by construction)
such that $\gamma\in \mathcal{M}_B(m_0)$ and $\gamma$ dominates $\gamma_0$.
Moreover $\gamma$ solves $(\mathcal{P}^*)$.
\end{lem}

\begin{proof}
By the Banach-Alaoglu-Bourbaki theorem, one may indeed assume that $\gamma_\eps$ weakly star converges to some $\gamma$. Obviously, $\gamma$ is $w$-comonotone and $\Pi_{\Sigma} \gamma=\Pi_{\Sigma } \gamma_0=m_0$, hence $\gamma \in \M_B(m_0)$. The fact that $\gamma$ dominates $\gamma_0$ is obtained by letting $\eps\to 0^+$ in \pref{almostdominance}. It remains to prove that $\gamma$ solves $\PPst$. Defining $\varphi^\eps:=\psi_\eps-w$ as in Lemma \ref{ekelandvp}, one has
\[J(\psi_\eps)=\int_{B^p} \sum_{i=1}^p \varphi_i^{\eps}(x_i) d (\gamma_0-\gamma_\eps)+\int_{B^p} \eta d(\gamma_\eps-\gamma_0) \to \inf \Qrim \mbox{ as } \eps \to 0^+.\]
By \pref{slackapprox}, passing to the limit thus yields
\[\inf \Qrim \leq \int_{B^p} \eta d(\gamma-\gamma_0),\]
which, combined with Lemma \ref{dualite11}, gives
\[ \int_{B^p} \eta  d \gamma\geq \inf \Qrim + \int_{B^p} \eta d\gamma_0=\max \PPst.\]
\end{proof}

\begin{lem}
\label{dominstrict} Let $\gamma$ be as in lemma \ref{lemcvgg}. Then:
\begin{enumerate}
\item if $\gamma_0$ solves $(\mathcal{P}^*)$ then $\gamma_0$ is $w$%
-comonotone,

\item $\gamma$ strictly dominates $\gamma_0$ unless $\gamma_0$ is
itself $w$-comonotone.
\end{enumerate}
\end{lem}

\begin{proof}
If $\gamma_0$ solves $\PPst$, it follows from Lemma \ref{dualite11} that $\inf \Qrim=0$. For any minimizing sequence $\psi_\eps$ (not necessarily the one constructed in Lemma \ref{ekelandvp})  of $\Qrim$, the following holds:
\[\begin{split}
0=\lim_{\eps\to 0^+}  J(\psi_\eps)=\lim_{\eps\to 0^+}  \int_{B^p} \Big(\sum_{i=1}^p \psi_{i, \eps} (x_i)- \Box_i \psi_{i, \eps} \Big(\sum_{i=1}^p x_i\Big) \Big) d\gamma_0(x)\\
= \lim_{\eps\to 0^+}  \int_{B^p} \Big(\sum_{i=1}^p \psi_{i, \eps} (x_i)-\sum_{i=1}^p \psi_{i, \eps} \Big(T_{\psi_\eps}^i\Big(\sum_{i=1}^p x_i\Big)\Big) \Big) d\gamma_0(x).
\end{split}\]
By a density argument, one may consider $\psi_\eps$ a minimizing sequence such that each $\psi_\eps$ belongs to $C^1(B)$. Fix $(x_1,..., x_p)$ and set $x:=\sum_{i=1}^p x_i$. Then $y^\eps:= T_{\psi_\eps}(x)$ can be characterized by the fact that there is a vector $p\in \R^d$ and a vector $(\lambda_i)$ of nonnegative weights such that
\begin{equation}\label{kktpsi}
\nabla \psi_{i, \eps}(y_i^\eps)=p-\lambda_i y_i^\eps, \; \lambda_i =0 \mbox{ if $y_i^{\eps}\notin \partial B$}, \; \sum_{i=1}^p y_i^{\eps} = x.
\end{equation}
On the other hand, since $w_i$ is strictly convex and $\psi_{i,\eps}-w_i \in \CC$, it follows that for any $a$ and $b$ in $B^2$,
\begin{equation}\label{modulpsi}
\psi_{i, \eps}(b)-\psi_{i, \eps} (a) \geq \nabla \psi_{i, \eps}(a) \cdot (b-a)+\theta_i(\vert b-a \vert)
\end{equation}
where function $\theta_i$ is defined by 
\[\theta_i(t):=\inf \{ w_i(b)-w_i(a)-\nabla w_i(a) \cdot(b-a), \; (a,b)\in B^2, \; \vert a-b \vert \geq t\}\]
for any $t\in [0, {\rm{diam}}(B)]$. Function $\theta_i$ (the modulus of strict convexity of $w_i$) is a  nondecreasing function such that $\theta_i(0)=0$ and $\theta_i(t)>0$ for $t>0$. Combining \pref{kktpsi} and \pref{modulpsi} yields
\[\begin{split}
\sum_{i=1}^p \psi_{i, \eps} (x_i)-\sum_{i=1}^p \psi_{i, \eps} (y_i^\eps)\geq \sum_{i=1}^p   \nabla \psi_{i, \eps} (y_i^\eps) \cdot (x_i-y_i^{\eps}) + \sum_{i=1}^p \theta_i (\vert x_i-y_i^{\eps}\vert) \\
= p\cdot  \sum_{i=1}^p (x_i-y_i^{\eps}) -\sum_{i=1}^p \lambda_i y_i^\eps (x_i -y_i^\eps) +   \sum_{i=1}^p \theta_i (\vert x_i-y_i^{\eps}\vert)\\
\geq   \sum_{i=1}^p \theta_i (\vert x_i-y_i^{\eps}\vert).
\end{split}\]
Hence the fact that $J(\psi_\eps)\to 0$ as $\eps\to 0^+$  implies
\[\lim_{\eps\to 0^+} \int_{B^p} \sum_{i=1}^p  \theta_i (\vert x_i-T_{\psi_\eps}^i(\sum_j x_j) \vert) d\gamma_0(x)=0\]
so that
\[T_{\psi_\eps}(\sum_j x_j)-x \to 0 \mbox{ as $\eps\to 0^+$ for $\gamma_0$-a.e. $x$}.\]
Therefore, by Lebesgue's dominated convergence theorem,
\[\begin{split}
\int_{B^p}  f(x) d\gamma_0(x) &= \lim_{\eps \to 0^+} \int_{B^p} f(T_{\psi_{\eps}}(\sum_j x_j )) d\gamma_0(x)
=\lim_{\eps \to 0^+} \int_{pB} f(T_{\psi_{\eps}}(x )) dm_0(x)\\
&= \lim_{\eps \to 0^+} \int_{B^p} f  d\gamma_{\psi_\eps}
\end{split}\]
holds for all $f\in C(B^p)$. Hence, $\gamma_{\psi_\eps}$ weakly star converges to $\gamma_0$ which proves that $\gamma_0$ is $w$-comonotone and Point $1$ is proven. We now prove Point $2$. If $\gamma_0$ is not $w$-comonotone then by Point $1$, it does not solve $\PPst$ and thus $\int \eta d (\gamma-\gamma_0) >0$ so that
\[\int_{B^p} \sum_{i=1}^p w_i(x_i) d\gamma < \int_{B^p} \sum_{i=1}^p w_i(x_i) d\gamma_0\]
hence $\gamma$ strictly dominates $\gamma_0$. This completes the proof.

\end{proof}

\subsection{Proof of Theorem \ref{dominance1d}}

Let $\gamma_0:=\LL(\xx)=\LL(X_1,..., X_p)$ and consider again the minimization problem $\PPst$. It follows from Lemma \ref{lemcvgg} that there exists a $w$-comonotone solution $\gamma$ to $\PPst$. By construction, $\gamma$ dominates $\gamma_0$, and  by the definition of $w$-comontonicity, there is a sequence $\gamma_n$ of $w$-strictly comonotone allocations that weakly star converges to $\gamma$. Lemma \ref{caractloi} allows to write $\gamma_n=\LL(\yy_n)$ for some $\yy_n\in \A(X)$ which is obviously comonotone in the univariate sense according to Definition \ref{defcomon1d}. Since $\yy_n$ is bounded in $L^{\infty}$ and each $\yy_n$ is a $1$-Lipschitz function of $X$, one may assume that $\yy_n$ converges uniformly (up to a subsequence) to some $\yy\in \A(X)$. It is obvious that $\yy$ is also comonotone and that $\gamma=\LL(\yy)$, hence $\yy$ dominates $\xx$. The strict dominance assertion follows from Lemma \ref{dominstrict}.


\begin{thebibliography}{99}
\bibitem{Arrow} K.J. Arrow, Uncertainty and the welfare of medical care,
Amer. Econom. Rev. 53 (1963), 941--973.

\bibitem{Arrow2} K.J. Arrow, Optimal insurance and generalized deductibles,
Scand. Actuar. J. (1974), 1--42.


\bibitem{Atkinson} A.B. Atkinson,  On the measurement of Inequality,
  Journal of Economic Theory 2 (1970) 244--263.

\bibitem{Attanasio} Attanasio, O., Davis, S., Relative Wage Movements and the Distribution of Consumption, Journal of Political Economy, 104(6), pp. 1227-62.

\bibitem{aubinek} J.-P. Aubin, I. Ekeland, Applied Nonlinear Analysis, John
Wiley ans Sons, New-York, 1984.

\bibitem{Borch} K. Borch, Equilibrium in a reinsurance market, Econometrica,
30 (1962), 424--444.

\bibitem{brenier} Y. Brenier, Polar factorization and monotone rearrangement
of vector-valued functions, Comm. Pure Appl. Math. 44 (1991), 375--417.

\bibitem{brown} Brown, D.J. and R.L. Matzkin, Testable Restrictions on the Equilibrium Manifold, Econometrica, Vol. 64 (1996), No. 6, p. 1249-1262.


\bibitem{cd2} G. Carlier, R.-A. Dana, Law invariant concave utility
functions and optimization problems with monotonicity and comonotonicity
constraints, Statistics and Decisions, 24 (2006), 127--152.

\bibitem{cd3} G. Carlier, R.-A. Dana, Two-Persons Efficient Risk-Sharing and
Equilibria for Concave Law-Invariant Utilities, Economic Theory, 36 (2008),
189--223.


\bibitem{cfm} P. Cartier, J.M.G. Fell, P.-A. Meyer, Comparaison des mesures port\'ees par un ensemble convexe compact, Bull. Soc. Math. France, 92 (1964), 435--445.


\bibitem{Chew}  S.H. Chew,  I. Zilcha, Invariance of the efficient sets
when the expected utility hypothesis is relaxed,  Journal of
Economic   Behaviour and Organisation  13   (1990)
125--131.




\bibitem{dana} R.-A. Dana, Market behavior when preferences are generated by second order stochastic dominance,   Journal of  Mathematical
Economics 40, (2004),     619-639.

\bibitem{DanaMeilijson} R.-A. Dana, I.I Meilijson, Modelling Agents'
Preferences in Complete Markets by Second Order Stochastic Dominance.
Working Paper 0333, Cahiers du CEREMADE
\bibitem{denneberg} D. Denneberg, Non-additive Measures and Integral, Kluwer
Academic Publishers, Holland, 1994.

\bibitem{Dybvig}  P. Dybvig,   Distributional Analysis of Portfolio Choice,
  Journal of Business 61 (1988)  369--393.

\bibitem{Ekeland} I. Ekeland, On the variational principle. J. Math. Anal.
Appl. 47 (1974) 324--353.

\bibitem{Ektem} I. Ekeland, R. Temam, Convex analysis and variational
problems, North-Holland, Amsterdam-Oxford, 1976.

\bibitem{GH11} A. Galichon, M. Henry, Dual theory of choice with multivariate risks, forthcoming in the Journal of Economic Theory.

\bibitem{EGH09} I. Ekeland, A. Galichon, M. Henry, Comonotonic measures of
multivariate risk, forthcoming in Mathematical Finance.
\bibitem{FolSch}H.  F\"{o}llmer, A. Schied,   Stochastic finance. An
introduction in discrete time ,
De Gruyter editor, Berlin  2004.

\bibitem{gs} W. Gangbo, A. \'Swi\c ech, Optimal maps for the multidimensional
Monge-Kantorovich problem, Comm. Pure Appl. Math., 51 (1998), 23--45.

\bibitem{jn1} E. Jouini, C. Napp, Comonotonic processes, Insurance Math.
Econom. 32 (2003), 255--265.

\bibitem{jn2} E. Jouini, C. Napp, Conditional comonotonicity, Decis. Econ.
Finance 27 (2004), 153--166.

\bibitem{Jouini}  E. Jouini ,  H. Kallal, (2000) Efficient Trading Strategies in
the  Presence of Market Frictions, Review of
Financial Studies 14 (2000)  343--369.


\bibitem{jst} E. Jouini, W. Schachermayer, N. Touzi, Optimal risk sharing
for law invariant monetary utility functions, Math. Finance 18 (2008),
269--292.

\bibitem{Landsberger} M. Landsberger, I.I. Meilijson, Comonotone
allocations, Bickel Lehmann dispersion and the Arrow-Pratt measure of risk
aversion, Annals of Operation Research 52 (1994) 97--106.

\bibitem{leroy} S.F. LeRoy,  J. Werner,  Principles of Financial
Economics.  Cambridge University  Press, Cambridge,  2001.


\bibitem{lyapunov} A. Lyapunov, Sur les fonctions-vecteurs completement additives, Bull. Acad. Sci. URSS (6) (1940), 465-478.


\bibitem{Rusch} M. Ludkovski, L. R\"{u}schendorf, On comonotonicity of Pareto
optimal risk sharing, Statistics and Probability Letters 78 (2008),
1181-1188.

\bibitem{Mul}  A. M\"{u}ller,     D. Stoyan,   Comparaisons Methods for
Stochastic Models and Risks , Wiley, New-York,  2002.




\bibitem{Peleg} B. Peleg,  M.E. Yaari,   A Price Characterisation of
Efficient Random Variables,   Econometrica 43  (1975)
283--292.



\bibitem{scarsini} G. Puccetti, M. Scarsini, Multivariate comonotonicity,
Journal of Multivariate Analysis 101 (2010), 291--304.

\bibitem{Rothschild} M. Rothschild, J. E. Stiglitz,. Increasing Risk, I. A
Definition, Journal of Economic Theory 2 (1970) 225--243.

\bibitem{Ru06} L. R\"{u}schendorf, Law invariant convex risk measures for portfolio vectors,  Statist. Decisions  24  (2006),  97--108.





\bibitem{strassen} V. Strassen, The existence of probability measures with given marginals,  Ann. Math. Statist.  36  (1965) 423--439.


\bibitem{townsend}  R.M. Townsend,  Risk and Insurance in village India, Econometrica 62 (1994)
539-592.

\bibitem{wilson}  R. Wilson,  The theory of syndicates, Econometrica 55 (1968),
95-115.

\end{thebibliography}
\end{document}